\documentclass[12pt,reqno]{amsart}
\usepackage{amsfonts}
\usepackage{amssymb,color}
\usepackage{amssymb}

\usepackage[colorlinks=true]{hyperref}
\hypersetup{urlcolor=blue, citecolor=blue}

\usepackage[margin=3cm, a4paper]{geometry}

\newtheorem{theo}{Theorem}[section]
\newtheorem{lemm}[theo]{Lemma}

\newtheorem{prop}[theo]{Proposition}
\newtheorem{rema}[theo]{Remark}
\numberwithin{equation}{section}

\newcommand{\bal}{\begin{align}}
\newcommand{\bbal}{\begin{align*}}
\newcommand{\beq}{\begin{equation}}
\newcommand{\eeq}{\end{equation}}
\newcommand{\bca}{\begin{cases}}
\newcommand{\eca}{\end{cases}}

\newcommand{\pa}{\partial}

\newcommand{\na}{\nabla}
\newcommand{\De}{\Delta}

\newcommand{\cd}{\cdot}

\newcommand{\dd}{\mathrm{d}}

\newcommand{\R}{\mathbb{R}}

\newcommand{\D}{\mathrm{div}}

\newcommand{\Div}{\mathrm{div\,}}
\newcommand{\p}{\partial}
\newcommand{\prq}{\\ &\quad}
\newcommand{\les}{{\lesssim}}
\newcommand{\pr}{\\ &}

\begin{document}


\title[A note on the inviscid limit of the incompressible MHD equations]{A note on the inviscid limit of the incompressible MHD equations}

\author[J. Li]{Jinlu Li}
\address{Department of Mathematics, Sun Yat-sen University, Guangzhou, 510275, China}
\email{lijl29@mail2.sysu.edu.cn}

\author[Z. Yin]{Zhaoyang Yin}
\address{Department of Mathematics, Sun Yat-sen University, Guangzhou, 510275, China \& Faculty of Information Technology,
 Macau University of Science and Technology, Macau, China}
\email{mcsyzy@mail.sysu.edu.cn}

\begin{abstract}
In this paper, we prove that as the viscosity and resistivity go to zero, the solution of the Cauchy problem for the incompressible MHD equations converges to the solution of the ideal MHD equations in the same topology with the initial data. Our proof mainly depends on the method introduced by the paper \cite{G.L.Y} and the constructions of the incompressible MHD equations.
\end{abstract}

\maketitle

\section{Introduction and main result}

In the paper, we consider the following Cauchy problem of the incompressible MHD equations:
\beq\label{MHD1}\begin{cases}
\partial_tu+u\cdot\nabla u-\mu\De u+\nabla P=b\cdot \nabla b, \\
\partial_tb+u\cdot \nabla b-\nu\De b=b\cdot\nabla u,\\
\mathrm{div} u=\mathrm{div} b=0,\quad (u,b)|_{t=0}=(u_0,b_0),
\end{cases}\eeq
where the unknowns are the vector fields $u: \R\times \R^d\rightarrow \R^d$, $b: \R\times \R^d\rightarrow \R^d$ and the scalar function $P$. Here, $u$ and $b$ are the velocity and magnetic, respectively, while $P$ denotes the pressure, $\mu\geq0$ is the viscosity coefficient and $\nu\geq0$ is the
magnetics diffusive coefficient. We expect the above equations converge to the following ideal MHD equations in the corresponding Besov spaces of the initial data when $\mu$ and $\nu$ tend to 0:
\beq\label{ideal-MHD}\begin{cases}
\partial_tu+u\cdot\nabla u+\nabla P=b\cdot \nabla b, \\
\partial_tb+u\cdot \nabla b=b\cdot\nabla u,\\
\mathrm{div} u=\mathrm{div} b=0,\quad (u,b)|_{t=0}=(u_0,b_0).
\end{cases}\eeq

The MHD system is a well-known model which governs the dynamics of the velocity and magnetic fields in electrically
conducting fluids such as plasmas, liquid metals, and salt water. The vanishing viscosity limit problem is one of the challenging topics in fluid dynamics. It has been studied by many authors in $H^s$ space (see \cite{A.D,D.L,S,W,X.X.W,Z}), where the main approaches were the energy arguments and depending strongly on the treatments of the trilinear forms which lead to the $(\mu,\nu)$ independent estimates. Most of their paper just prove in the lower regularity comparing the topology for the initial data. Therefore, we want to solve the inviscid limit of the incompressible MHD equations in the same topology with the initial data. Our proof mainly relies on the method introduced by the paper \cite{G.L.Y} and the constructions of the incompressible MHD equations.

Then, our argument can state as follows:
\begin{theo}\label{th1}
Let $d\geq 2$ and $s>1+\frac d2$. Suppose that $u^n_0\in H^s(\R^d)$ goes to $u_0\in H^s(\R^d)$ in $H^s(\R^d)$ and $b^n_0\in H^s(\R^d)$ goes to $b_0\in H^s(\R^d)$ in $H^s(\R^d)$ when $n$ goes to infinity. Let $\mu_n,\nu_n\geq0$ and $\mu_n,\nu_n$ go to 0 when $n$ goes to $\infty$. If $\D u_0=\D b_0=\D u^n_0=\D b^n_0=0$, then there exists a positive $T>0$ independent of $n$ such that  $(u^n,b^n)\in \mathcal{C}([0,T];H^s(\R^d))$ be the solution of
\bal\label{revist1-MHD}\begin{cases}
\partial_tu^n+u^n\cdot\nabla u^n-\mu_n\De u^n+\nabla P_n=b^n\cdot \nabla b^n, \\
\partial_tb^n+u^n\cdot \nabla b^n-\nu_n\De b^n=b^n\cdot\nabla u^n,\\
\mathrm{div} u^n=\mathrm{div} b^n=0,\quad (u^n,b^n)|_{t=0}=(u^n_0,b^n_0),
\end{cases}\end{align}
and $(u,b)\in \mathcal{C}([0,T];H^s(\R^d))$ be the solution of \eqref{ideal-MHD} with initial data $(u_0,b_0)$. Moreover, there holds
\bbal
\lim_{n\rightarrow \infty}||u^n-u||_{L^\infty_T(H^s(\R^d))}=0.
\end{align*}
\end{theo}

If the viscosity coefficient is equal to the magnetics diffusive coefficient, that is $\nu=\mu$, then the system \eqref{MHD1} becomes
\beq\label{MHD2}\begin{cases}
\partial_tu+u\cdot\nabla u-\mu\De u+\nabla P=b\cdot \nabla b, \\
\partial_tb+u\cdot \nabla b-\mu\De b=b\cdot\nabla u,\\
\mathrm{div} u=\mathrm{div} b=0,\quad (u,b)|_{t=0}=(u_0,b_0).
\end{cases}\eeq
Then, we have a generalized conclusion as follows:
\begin{theo}\label{th2}
Let $d\geq 2$. Assume that $(s,p,r)$ satisfies
\begin{align}\label{eq:spr}
s>\frac{d}{p}+1, p\in [1,\infty], r\in (1,\infty) \quad \mathrm{or} \quad
s=\frac{d}{p}+1, p\in [1,\infty], r=1.
\end{align}
Suppose that $u^n_0\in B^s_{p,r}(\R^d)$ goes to $u_0\in B^s_{p,r}(\R^d)$ in $B^s_{p,r}(\R^d)$ when $n$ goes to infinity.  Let $\mu_n\geq0$ and $\mu_n$ goes to 0 when $n$ goes to $\infty$. If $\D u_0=\D b_0=\D u^n_0=\D b^n_0=0$, then there exists a positive $T>0$ independent of $n$ such that  $(u^n,b^n)\in \mathcal{C}([0,T];B^s_{p,r}(\R^d))$ be the solution of
\bal\label{revist2-MHD}\begin{cases}
\partial_tu^n+u^n\cdot\nabla u^n-\mu_n\De u^n+\nabla P_n=b^n\cdot \nabla b^n, \\
\partial_tb^n+u^n\cdot \nabla b^n-\mu_n\De b^n=b^n\cdot\nabla u^n,\\
\mathrm{div} u^n=\mathrm{div} b^n=0,\quad (u^n,b^n)|_{t=0}=(u^n_0,b^n_0),
\end{cases}\end{align}
and $(u,b)\in \mathcal{C}([0,T];B^s_{p,r}(\R^d))$ be the solution of \eqref{ideal-MHD} with initial data $(u_0,b_0)$. Moreover, there holds
\bbal
\lim_{n\rightarrow \infty}||u^n-u||_{L^\infty_T(B^s_{p,r}(\R^d))}=0.
\end{align*}
\end{theo}

\begin{rema}
We prove local well-posedness in $B^s_{p,r}$ where $(s,p,r)$ satisfies \eqref{eq:spr}, improving the previous result in \cite{M.Y}.
\end{rema}

\noindent\textbf{Notations.} Given a Banach space $X$, we denote its norm by $\|\cdot\|_{X}$. Since all spaces of functions are over $\mathbb{R}^d$, for simplicity, we drop  $\mathbb{R}^d$ in our notations of function spaces if there is no ambiguity. The symbol $A\lesssim B$ denotes that there exists a constant $\bar{c}_0>0$ independent of $A$ and $B$, such that $A\leq \bar{c}_0 B$. The symbol $A\simeq B$ represents $A\lesssim B$ and $B\lesssim A$.

\section{Preliminaries}

In this section we collect some preliminary definitions and lemmas. For more details we refer the readers to \cite{B.C.D}.

Let $\chi: {\mathbb R}^d\to [0, 1]$ be a radial, non-negative,
smooth and radially decreasing function which is supported in $\mathcal{B}\triangleq \{\xi:|\xi|\leq \frac43\}$ and
$\chi\equiv 1$ for $|\xi|\leq \frac54$. Let $\varphi(\xi)=\chi(\frac{\xi}{2})-\chi(\xi)$. Then $\varphi$ is supported in the ring $\mathcal{C}\triangleq \{\xi\in\mathbb{R}^d:\frac 3 4\leq|\xi|\leq \frac 8 3\}$.
For $u \in \mathcal{S}'$, $q\in {\mathbb Z}$, we define the Littlewood-Paley operators: $\dot{\Delta}_q{u}=\mathcal{F}^{-1}(\varphi(2^{-q}\cdot)\mathcal{F}u)$, ${\Delta}_q{u}=\dot{\Delta}_q{u}$ for $q\geq 0$, ${\Delta}_q{u}=0$ for $q\leq -2$ and $\Delta_{-1}u=\mathcal{F}^{-1}(\chi \mathcal{F}u)$, and $S_q{u}=\mathcal{F}^{-1}\big(\chi(2^{-q}\xi)\mathcal{F}u\big)$.
Here we use ${\mathcal{F}}(f)$ or $\widehat{f}$ to denote
the Fourier transform of $f$.

We define the standard vector-valued Besov spaces $B^s_{p,r}$ and $\dot B^s_{p,r}$ of the functions $u:{\mathbb R}^d\to {\mathbb R}^d$ with finite norms which are defined by
\begin{align*}
\|u\|_{B^s_{p,r}}&\triangleq \big|\big|(2^{js}\|\Delta_j{u}\|_{L^p})_{j\in {\mathbb Z}}\big|\big|_{\ell^r},\\
\|u\|_{\dot{B}^s_{p,r}}&\triangleq \big|\big|(2^{js}\|\dot{\Delta}_j{u}\|_{L^p})_{j\in {\mathbb Z}}\big|\big|_{\ell^r}.
\end{align*}

\begin{rema}\label{re3}
Let $s\geq0$ and $1\leq p,r\leq\infty$. There holds
\[||(\mathrm{Id}-\De_{-1})u||_{\dot{B}^s_{p,r}}\simeq||(\mathrm{Id}-\De_{-1})u||_{B^s_{p,r}};\]
\[||u||_{\dot{B}^s_{p,r}}\lesssim ||u||_{B^s_{p,r}}, \quad \mathrm{if} \quad s>0; \quad ||u||_{L^p}\lesssim ||u||_{B^0_{p,1}}\lesssim||u||_{\dot{B}^0_{p,1}}.\]
\end{rema}

Next we recall nonhomogeneous Bony's decomposition from \cite{B.C.D}.
$$uv=T_uv+T_vu+R(u,v),$$
with
\[T_uv\triangleq \sum\limits_j S_{j-1}u\Delta_j v, \quad R(u,v)\triangleq \sum_{j}\sum\limits_{|k-j|\leq1}\Delta_j u \Delta_k v.\]
This is now a standard tool for nonlinear estimates. Now we use Bony's decomposition to prove some nonlinear estimates which will be used for the estimate of pressure term.

\begin{rema}
Note that $S_{j-1}u\Delta_j v$ is supported in the $2^j\mathcal{\tilde{C}}$ where the annulus $\mathcal{\tilde{C}}\triangleq \{\xi\in\mathbb{R}^d:\frac{1}{12}\leq|\xi|\leq \frac{12}{3}\}$.
\end{rema}

\begin{rema}\label{re5}
Let $f$ be a smooth function on $\R^d\backslash \{0\}$ which satisfies $f(\lambda\xi)=\lambda^mf(\xi)$. Then, we have
\[||f(D)u||_{\dot{B}^{s-m}_{p,r}}\leq C||u||_{\dot{B}^s_{p,r}},\]
and
\[||f(D)T_uv||_{B^{s-m}_{p,r}}\leq C||T_uv||_{B^s_{p,r}}.\]
\end{rema}

The main properties of the paraproduct and remainder are described below.
\begin{lemm}\label{le2.1}
Let $s\in\mathbb{R}$ and $1\leq p,r\leq\infty$. Then there exists a constant $C$, depending only on $d,p,r,s$, such that
$$ \|T_uv\|_{B^{s}_{p,r}}\leq C||u||_{L^\infty}||v||_{B^s_{p,r}}.$$
Let $s>0$ and $1\leq p,r\leq\infty$. Then there exists a constant $C$, depending only on $d,p,r,s$, such that
\[\|R(u,v)\|_{B^{s}_{p,r}}\leq C\|u\|_{L^\infty}\|v\|_{B^{s}_{p,r}},\]
or
\[\|R(u,v)\|_{B^{s}_{p,r}}\leq C\|u\|_{B^{-1}_{\infty,r}}\|v\|_{B^{s+1}_{p,\infty}}.\]
Let $s>1$ and $1\leq p,r\leq\infty$. Then there exists a constant $C$, depending only on $d,p,r,s$, such that
\[\|R(u,v)\|_{B^{s}_{p,r}}\leq  C(||u||_{L^\infty}+||\na u||_{L^\infty})\|v\|_{B^{s-1}_{p,r}}.\]
\end{lemm}

\begin{lemm}
Let $s>0$ and $1\leq p,r\leq\infty$. Then there exists a constant $C$, depending only on $d,p,r,s$, such that
\[||uv||_{B^s_{p,r}}\leq C\big(||u||_{L^\infty}||v||_{B^s_{p,r}}+||v||_{L^\infty}||u||_{B^s_{p,r}}\big).\]
\end{lemm}

\begin{lemm}\label{le3.1}
Assume $(\sigma,p,r)$ satisfies \eqref{eq:spr}. Then there exists a constant $C$, depending only on $d,p,r,\sigma$, such that for all $u,f\in B^\sigma_{p,r}$ with $\Div u=0$,
\[||u\cdot \na f||_{B^{\sigma-1}_{p,r}}\leq C||u||_{B^{\sigma-1}_{p,r}}||f||_{B^\sigma_{p,r}}.\]
\end{lemm}
\begin{proof}
Due to $\Div u=0$, then we decompose the term $u\cdot \na f$ into
\[u\cdot \na f=\sum^d_{i=1}T_{u^i}\p_if+\sum^d_{i=1}T_{\p_if}u^i+\sum^d_{i=1}\p_iR(u^i,f).\]
According to Lemma \ref{le2.1}, we have
\[||T_{u^i}\p_if||_{B^{\sigma-1}_{p,r}}\leq C||u||_{L^\infty}||\na f||_{B^\sigma_{p,r}}\leq C||u||_{B^{\sigma-1}_{p,r}}||f||_{B^\sigma_{p,r}},\]
\[||T_{\p_if}u^i||_{B^{\sigma-1}_{p,r}}\leq C||\na f||_{L^\infty}||u||_{B^\sigma_{p,r}}\leq C||u||_{B^{\sigma-1}_{p,r}}||f||_{B^\sigma_{p,r}},\]
\begin{align*}
||\p_iR(u^i,f)||_{B^{\sigma-1}_{p,r}}&\leq C||R(u^i,f)||_{B^{\sigma}_{p,r}}\leq C||u||_{L^\infty}||f||_{B^\sigma_{p,r}}\\&\leq C||u||_{B^{\sigma-1}_{p,r}}||f||_{B^\sigma_{p,r}}.
\end{align*}
Combining this results, we complete the proof of this lemma.
\end{proof}

\begin{lemm}\label{le:P1}
Assume $(\sigma,p,r)$ satisfies \eqref{eq:spr}. Then there exists a constant $C$, depending only on $d,p,r,\sigma$, such that for all $u,v\in B^\sigma_{p,r}$ with $\Div u=\Div v=0$,
\begin{align*}
||\na(-\De)^{-1}\Div(u\cdot \na v)||_{B^\sigma_{p,r}}&\leq C \big((||u||_{L^\infty}+||\na u||_{L^\infty})||v||_{B^\sigma_{p,r}}
\prq +(||v||_{L^\infty}+||\na v||_{L^\infty})||u||_{B^\sigma_{p,r}}\big).
\end{align*}
\end{lemm}
\begin{proof}
It is easy to get
\begin{align*}
& \quad ||\na(-\De)^{-1}\Div(u\cdot \na v)||_{B^\sigma_{p,r}}\\&\les  ||\na(-\De)^{-1}\Div(\mathrm{Id}-\De_{-1})(u\cdot \na v)||_{\dot{B}^\sigma_{p,r}}+||\na(-\De)^{-1}\Div(u\cdot \na v)||_{L^p}\pr
\les  ||\na(-\De)^{-1}(\mathrm{Id}-\De_{-1})(\na u\cdot \na v)||_{\dot{B}^\sigma_{p,r}}+||\na(-\De)^{-1}\Div\Div(u\otimes v)||_{L^p}
\triangleq \mathrm{I}+\mathrm{II}.
\end{align*}
According to Remark \ref{re3} and Lemma \ref{le2.1}, we have
\begin{align*}
\mathrm{I}&\lesssim ||\sum^d_{i,j=1}\na(-\De)^{-1}(\mathrm{Id}-\De_{-1})[T_{\p_j u^i}(\p_iv^j)+ T_{\p_i v^j}(\p_ju^i)]||_{\dot{B}^\sigma_{p,r}}\prq
+||\sum^d_{i,j=1}\na(-\De)^{-1}\p_i\p_j (\mathrm{Id}-\De_{-1})R(u^i,v^j)||_{\dot{B}^\sigma_{p,r}}\pr
\les \sum^d_{i,j=1}||(\mathrm{Id}-\De_{-1})[T_{\p_j u^i}(\p_iv^j)+ T_{\p_i v^j}(\p_ju^i)]||_{\dot{B}^{\sigma-1}_{p,r}}\prq
+\sum^d_{i,j=1}||(\mathrm{Id}-\De_{-1})R(u^i,v^j)||_{\dot{B}^{\sigma+1}_{p,r}} \pr
\les \sum^d_{i,j=1}||[T_{\p_j u^i}(\p_iv^j)+ T_{\p_i v^j}(\p_ju^i)]||_{B^{\sigma-1}_{p,r}}+\sum^d_{i,j=1}||R(u^i,v^j)||_{B^{\sigma+1}_{p,r}}
\pr \les||\na u||_{L^\infty}||v||_{B^\sigma_{p,r}}+(||v||_{L^\infty}+||\na v||_{L^\infty})||u||_{B^\sigma_{p,r}},
\end{align*}
and
\begin{align*}
\mathrm{II}&\les ||\na(-\De)^{-1}\Div \Div(u\otimes v)||_{\dot{B}^0_{p,1}}
\les ||u\otimes v||_{\dot{B}^1_{p,1}} \les ||u\otimes v||_{B^1_{p,1}}
\pr \les ||u\otimes v||_{B^\sigma_{p,r}}\les ||u||_{L^\infty}||v||_{B^\sigma_{p,r}}+||v||_{L^\infty}||u||_{B^\sigma_{p,r}}.
\end{align*}
This completes the proof of this lemma.
\end{proof}

\begin{lemm}\label{le:P2}
Assume $(\sigma,p,r)$ satisfies \eqref{eq:spr}. Then there exists a constant $C$, depending only on $d,p,r,\sigma$, such that for all $u,v\in B^\sigma_{p,r}$ with $\Div u=\Div v=0$,
\begin{align*}
||\na(-\De)^{-1}\Div(u\cdot \na v)||_{B^{\sigma-1}_{p,r}}&\leq C ||u||_{B^{\sigma-1}_{p,r}}||v||_{B^\sigma_{p,r}},
\end{align*}
or
\begin{align*}
||\na(-\De)^{-1}\Div(u\cdot \na v)||_{B^{\sigma-1}_{p,r}}&\leq C ||v||_{B^{\sigma-1}_{p,r}}||u||_{B^\sigma_{p,r}}.
\end{align*}
\end{lemm}
\begin{proof}
Due to the fact that $\Div(u\cdot \na v)=\Div(v\cdot \na u)$, then we have
\begin{align*}
& \quad ||\na(-\De)^{-1}\Div(u\cdot \na v)||_{B^{\sigma-1}_{p,r}}\\&\les  ||\na(-\De)^{-1}\Div(\mathrm{Id}-\De_{-1})(u\cdot \na v)||_{\dot{B}^{\sigma-1}_{p,r}}+||\na(-\De)^{-1}\Div\De_{-1}(u\cdot \na v)||_{L^p}\pr
\les  ||(\mathrm{Id}-\De_{-1})(u\cdot \na v)||_{\dot{B}^{\sigma-1}_{p,r}}+||\na(-\De)^{-1}\Div(u\cdot \na v)||_{L^p}
\triangleq \mathrm{III}+\mathrm{IV}.
\end{align*}
Due to Remarks \ref{re3}-\ref{re5}, Lemma \ref{le2.1} and Lemma \ref{le3.1}, we get
\begin{align*}
\mathrm{III} \les  ||(\mathrm{Id}-\De_{-1})(u\cdot \na v)||_{B^{\sigma-1}_{p,r}}
\les  ||u\cdot \na v||_{B^{\sigma-1}_{p,r}}
\les ||u||_{B^{\sigma-1}_{p,r}}||v||_{B^{\sigma}_{p,r}},
\end{align*}
and
\begin{align*}
\mathrm{IV} &
\les ||\sum^d_{i,j=1}\na(-\De)^{-1}\p_i(T_{u^j}\p_j v^i+T_{\p_j v^i}u^j)||_{B^0_{p,1}}\prq
+||\sum^d_{i,j=1}\na(-\De)^{-1}\p_i\p_jR(u^i,u^j)||_{\dot{B}^0_{p,1}}\pr
\les \sum^d_{i,j=1}||T_{u^j}\p_j v^i+T_{\p_j v^i}u^j||_{B^0_{p,1}}+\sum^d_{i,j=1}||R(u^i,v^j)||_{\dot{B}^1_{p,1}}\pr
\les \sum^d_{i,j=1}||T_{u^j}\p_j v^i+T_{\p_j v^i}u^j||_{B^{\sigma-1}_{p,r}}+\sum^d_{i,j=1}||R(u^i,v^j)||_{B^1_{p,1}}\pr
\les ||u||_{B^{\sigma-1}_{p,r}}||\na v||_{L^\infty}+||u||_{L^\infty}||v||_{B^{\sigma}_{p,r}}+||u||_{L^\infty}||v||_{B^1_{p,1}}\pr
\les ||u||_{B^{\sigma-1}_{p,r}}||v||_{B^{\sigma}_{p,r}}.
\end{align*}
This completes this proof of this lemma.
\end{proof}

Combining the results of Lemmas \ref{le3.1}-\ref{le:P2}, we have
\begin{lemm}\label{lem:P}
Assume $(s,p,r)$ satisfies \eqref{eq:spr}. Then

1) there exists a constant $C$, depending only on $d,p,r,s$, such that for all $u,f\in B^s_{p,r}$ with $\mathrm{div\,} u=0$,
\[\|u\cdot \na f\|_{B^{s-1}_{p,r}}\leq C\|u\|_{B^{s-1}_{p,r}}\|f\|_{B^s_{p,r}}.\]

2) there exists a constant $C$, depending only on $d,p,r,s$, such that for all $u,v\in B^s_{p,r}$ with $\mathrm{div\,} u=\mathrm{div\,} v=0$,
\begin{align*}
\|\na(-\De)^{-1}\mathrm{div\,}(u\cdot \na v)\|_{B^s_{p,r}}&\leq C \big(\|u\|_{C^{0,1}}\|v\|_{B^s_{p,r}}+\|v\|_{C^{0,1}}\|u\|_{B^s_{p,r}}\big),\\
\|\na(-\De)^{-1}\mathrm{div\,}(u\cdot \na v)\|_{B^{s-1}_{p,r}}&\leq C \min(\|u\|_{B^{s-1}_{p,r}}\|v\|_{B^s_{p,r}},\|v\|_{B^{s-1}_{p,r}}\|u\|_{B^s_{p,r}}),
\end{align*}
where $||f||_{C^{0,1}}=||f||_{L^\infty}+||\na f||_{L^\infty}$.
\end{lemm}

\begin{lemm}\cite{L.Y}\label{le2}
Let $\sigma\in\mathbb{R}$ and $1\leq p,r\leq\infty$. Let $v$ be a vector field over $\mathbb{R}^d$. Assume that $\sigma>-d\min\{1-\frac{1}{p},\frac{1}{p}\}$. Define $R_j=[v\cdot\nabla,\Delta_j]f$. There a constant $C=C(p,\sigma,d)$ such that \\
\begin{equation*}
\big|\big|(2^{j\sigma}||R_j||_{L^p})_{j\geq-1}\big|\big|_{\ell^r}\leq
\begin{cases}
 C||\nabla v||_{B^{\frac dp}_{p,\infty}\cap L^\infty}||f||_{B^\sigma_{p,r}},\ \mathrm{if} \quad \sigma<1+\frac{d}{p},\\
 C||\nabla v||_{B^{\frac dp+1}_{p,\infty}}||f||_{B^{\sigma}_{p,r}},\quad \quad \mathrm{if} \quad \sigma=1+\frac{d}{p}\ ,r>1,\\
 C||\nabla v||_{B^{\sigma-1}_{p,r}}||f||_{B^\sigma_{p,r}}, \quad \quad \quad \quad  \mathrm{otherwise}.
\end{cases}\end{equation*}
\end{lemm}

We need an estimate for the transport-diffusion equation which is uniform with respect to the viscocity. Consider the following equation:
\beq\label{eq:TDep}
\begin{cases}
\pa_t f+v\cdot \nabla f-\varepsilon \Delta f=g,\\
f(0)=f_0,
\end{cases}
\eeq
where $\varepsilon\geq 0$, $v:{\mathbb R}\times {\mathbb R}^d \to {\mathbb R}^d$, $f_0:{\mathbb R}^d\to {\mathbb R}^N$, and $g:{\mathbb R}\times {\mathbb R}^d\to {\mathbb R}^N$ are given.

\begin{prop}\label{pr:TDe}
Let $d\geq 2, 1\leq p\leq \infty, 1\leq r\leq \infty$. Assume that $\D v=0$ and $s>-1$. There exists a constant $C$, depending only on $d,p,r,s$, such that for any smooth solution $f$ of \eqref{eq:TDep} we have
\begin{align}\label{ES2}
||f||_{L_t^\infty (B^{s}_{p,r})}&\leq ||f_0||_{B^s_{p,r}}
+\int^t_0||g(\tau)||_{B^{s}_{p,r}}\dd \tau
\nonumber\\& \quad +C\int^t_0(||\na v||_{L^\infty}||f||_{B^s_{p,r}}+||\na f||_{L^\infty}||v||_{B^s_{p,r}})\dd \tau.
\end{align}
\end{prop}
\begin{proof}
First, applying $\Delta_j$ to \eqref{eq:TDep} yields
\begin{equation}\label{yin}
\quad \partial_t \Delta_jf+v\cdot \nabla \Delta_jf-\varepsilon\De \De_jf=\Delta_jg+R_j,\quad \quad \quad \quad \quad \Delta_jf|_{t=0} =\Delta_jf_0.
\end{equation}
If $p\in[1,\infty)$, we multiply \eqref{yin} by $|\De_jf|^{p-1}\mathrm{sgn}(\De_jf)$ and then integrate by parts to get
\begin{align*}
\frac1p\frac{\dd}{\dd t}||\De_jf||^p_{L^p}-\varepsilon\int_{\R^d}\De\De_jf|\De_jf|^{p-2}\De_jf \dd x &\leq \frac1p\int_{\R^d} \D v |\De_jf|^p \dd x
\\&\quad +\int_{\R^d}R_j|\De_jf|^{p-2}\De_jf \dd x.
\end{align*}
Using the fact (see Proposition 2.1 in \cite{Danchin}) that for all $j\geq -1$,
\[-\varepsilon\int_{\R^d}\De\De_jf|\De_jf|^{p-2}\De_jf \dd x\geq 0,\]
we have
\begin{align}\label{yin1}
||\Delta_jf(t)||_{L^p}&\leq ||\Delta_jf_0||_{L^p}+\int^t_0||\Delta_jg(\tau)||_{L^p}\dd\tau\nonumber\\&\quad
+\int^t_0\big(||R_j(\tau)||_{L^p}+\frac1p||\mathrm{div} v(\tau)||_{L^\infty}||\Delta_jf(\tau)||_{L^p}\big)\dd\tau.
\end{align}
If $p=\infty$, the above inequality also holds by the maximum principle. According to the fact (see Lemma 2.100, \cite{B.C.D}), for $s>-1, \ \D v=0$, we have
\bal\label{Li-yin}
\big|\big|(2^{js}||R_j||_{L^p})_{j\geq-1}\big|\big|_{\ell^r}\leq C(||\nabla v||_{L^\infty}||f||_{B^s_{p,r}}+||\nabla f||_{L^\infty}||\nabla v||_{B^{s-1}_{p,r}}).
\end{align}
Multiplying both sides of \eqref{yin1} by $2^{js}$ and taking $\ell^r$ norm, it follows from Minkovshi's inequality that
\begin{align*}
||f(t)||_{B^s_{p,r}}&\leq ||f_0||_{B^s_{p,r}}+\int_0^t||g(\tau)||_{B^s_{p,r}}\dd\tau
\nonumber\\& \quad +C\int^t_0(||\na v||_{L^\infty}||f||_{B^s_{p,r}}+||\na f||_{L^\infty}||v||_{B^s_{p,r}})\dd \tau.
\end{align*}
This completes the proof of this lemma.
\end{proof}

\section{Some usefull estimates}

In this section, we will establish some usefull estimates for smooth solutions of \eqref{MHD1} and \eqref{MHD2}, which is the key component in the proof of Theorems \ref{th1} and \ref{th2}. This estimates can be stated as follows.

\begin{lemm}\label{le-1}
Let $d\geq 2$ and $s>1+\frac d2$. Suppose that $(u^1,b^1)\in \mathcal{C}([0,T];H^{s})$ and $(u^2,b^2)\in \mathcal{C}([0,T];H^{s+1})$ are two solutions of \eqref{MHD1} with initial data $(u^1_0,b^1_0)$ and $(u^2_0,b^2_0)$ respectively. Denote $\delta u=u^1-u^2$ and $\delta b=b^1-b^2$. Then, we have
\bbal
||\delta u(t)||^2_{H^{s-1}}+||\delta b(t)||^2_{H^{s-1}}\leq (||\delta u_0||^2_{H^{s-1}}+||\delta b_0||^2_{H^{s-1}})e^{\mathrm{A}(t)},
\end{align*}
and
\bbal
||\delta u(t)||^2_{H^{s}}+||\delta b(t)||^2_{H^{s}}&\leq \Big(||\delta u_0||^2_{H^{s}}+||\delta b_0||^2_{H^{s}}\\& \quad +C\int^t_0(||u^2||^2_{H^{s+1}}+||b^2||^2_{H^{s+1}})(||\delta u||^2_{H^{s-1}}+||\delta b||^2_{H^{s-1}}) \dd \tau\Big)e^{\mathrm{A}(t)},
\end{align*}
with
\bbal
\mathrm{A}(t)=C\int^t_0(1+||u^1||_{H^s}+||u^2||_{H^s}+||b^1||_{H^s}+||b^2||_{H^s})\dd \tau.
\end{align*}
\end{lemm}
\begin{proof}
It is easy to show that
\beq\label{MHD1-1}\bca
\partial_t\delta u+u^1\cdot\nabla \delta u+\delta u\cd\nabla u^2-\mu\De \delta u+\nabla P=b^1\cdot \nabla \delta b+\delta b\cd \na b^2, \\
\partial_t\delta b+u^1\cdot \nabla \delta b+\delta u\cd \na b^2-\nu\De \delta b=b^1\cdot\nabla \delta u+\delta b\cd \na u^2,\\
\mathrm{div}\delta u=\mathrm{div}\delta b=0,\quad (\delta u,\delta b)|_{t=0}=(\delta u_0,\delta b_0).
\eca\eeq
Now, we apply $\De_j$ to \eqref{MHD1-1}, and take the inner product with $(\De_j\delta u,\De_j\delta b)$ and integrate by parts to have
\bal\label{eq1}
\frac12\frac{\dd}{\dd t}(||\De_j\delta u||^2_{L^2}+||\De_j\delta b||^2_{L^2})\leq K_1+K_2+K_3+K_4+K_5,
\end{align}
where
\bbal
&K_1=-\int_{\R^d}[\De_j,u^1\cdot \na]\delta u\cd \De_j\delta u\ \dd x, \quad \quad K_2=-\int_{\R^d}[\De_j,u^1\cdot \na]\delta b\cd \De_j\delta b\ \dd x,\\
&K_3=\int_{\R^d}[\De_j,b^1\cdot \na]\delta b\cd \De_j\delta u\ \dd x, \quad \quad K_4=\int_{\R^d}[\De_j,b^1\cdot \na]\delta u\cd \De_j\delta b\ \dd x,\\
&K_5=\int_{\R^d}\De_j(\delta b\cd \na b^2-\delta u\cd\nabla u^2)\De_j\delta u\ \dd x,\\
&K_6=\int_{\R^d}\De_j(\delta b\cd \na u^2-\delta u\cd\nabla b^2)\De_j\delta b\ \dd x.
\end{align*}
On the one hand, according to Lemma \ref{le2}, it is easy to estimate
\bbal
&|K_1|\lesssim ||[\De_j,u^1\cdot \na]\delta u||_{L^2}||\De_j\delta u||_{L^2}\lesssim 2^{-2j(s-1)}c^2_j||\na u^1||_{H^{s-1}}||\delta u||^2_{H^{s-1}},\\
&|K_2|\lesssim ||[\De_j,u^1\cdot \na]\delta b||_{L^2}||\De_j\delta b||_{L^2}\lesssim 2^{-2j(s-1)}c^2_j||\na u^1||_{H^{s-1}}||\delta b||^2_{H^{s-1}},\\
&|K_3|\lesssim ||[\De_j,b^1\cdot \na]\delta b||_{L^2}||\De_j\delta u||_{L^2}\lesssim 2^{-2j(s-1)}c^2_j||\na b^1||_{H^{s-1}}||\delta u||_{H^{s-1}}||\delta b||_{H^{s-1}},\\
&|K_4|\lesssim ||[\De_j,b^1\cdot \na]\delta u||_{L^2}||\De_j\delta b||_{L^2}\lesssim 2^{-2j(s-1)}c^2_j||\na b^1||_{H^{s-1}}||\delta u||_{H^{s-1}}||\delta b||_{H^{s-1}},\\
&|K_5|\lesssim 2^{-2j(s-1)}c^2_j(||b^2||_{H^{s}}||\delta u||_{H^{s-1}}||\delta b||_{H^{s-1}}+||u^2||_{H^s}||\delta u||^2_{H^{s-1}}),\\
&|K_6|\lesssim 2^{-2j(s-1)}c^2_j(||b^2||_{H^{s}}||\delta u||_{H^{s-1}}||\delta b||_{H^{s-1}}+||u^2||_{H^s}||\delta b||^2_{H^{s-1}}).
\end{align*}
Integrating \eqref{eq1} over $[0,t]$, multiplying the inequality above by $2^{2j(s-1)}$ and summing over $j\geq -1$, we have
\bal\label{eq2}
||\delta u||^2_{H^{s-1}}+||\delta b||^2_{H^{s-1}}&\leq ||\delta u_0||^2_{H^{s-1}}+||\delta b_0||^2_{H^{s-1}}\nonumber\\& \quad +\int^t_0\mathrm{A}'(\tau)(||\delta u||^2_{H^{s-1}}+||\delta b||^2_{H^{s-1}})\dd \tau.
\end{align}
On the anther hand, according to Lemma \ref{le2}, it is easy to estimate
\bbal
&|K_1|\lesssim ||[\De_j,u^1\cdot \na]\delta u||_{L^2}||\De_j\delta u||_{L^2}\lesssim 2^{-2js}c^2_j||\na u^1||_{H^{s-1}}||\delta u||^2_{H^{s}},\\
&|K_2|\lesssim ||[\De_j,u^1\cdot \na]\delta b||_{L^2}||\De_j\delta b||_{L^2}\lesssim 2^{-2js}c^2_j||\na u^1||_{H^{s-1}}||\delta b||^2_{H^{s}},\\
&|K_3|\lesssim ||[\De_j,b^1\cdot \na]\delta b||_{L^2}||\De_j\delta u||_{L^2}\lesssim 2^{-2js}c^2_j||\na b^1||_{H^{s-1}}||\delta u||_{H^{s}}||\delta b||_{H^{s}},\\
&|K_4|\lesssim ||[\De_j,b^1\cdot \na]\delta u||_{L^2}||\De_j\delta b||_{L^2}\lesssim 2^{-2js}c^2_j||\na b^1||_{H^{s-1}}||\delta u||_{H^{s}}||\delta b||_{H^{s}},\\
&|K_5|\lesssim 2^{-2js}c^2_j(||b^2||_{H^{s}}||\delta u||_{H^{s}}||\delta b||_{H^{s}}+||u^2||_{H^s}||\delta u||^2_{H^{s}}\\& \quad \quad
\quad +||b^2||_{H^{s+1}}||\delta u||_{H^{s}}||\delta b||_{H^{s-1}}+||u^2||_{H^{s+1}}||\delta u||_{H^{s}}||\delta u||_{H^{s-1}}),\\
&|K_6|\lesssim 2^{-2js}c^2_j(||b^2||_{H^{s}}||\delta u||_{H^{s}}||\delta b||_{H^{s}}+||u^2||_{H^s}||\delta b||^2_{H^{s}}\\& \quad \quad
\quad +||b^2||_{H^{s+1}}||\delta u||_{H^{s-1}}||\delta b||_{H^{s}}+||u^2||_{H^{s+1}}||\delta b||_{H^{s}}||\delta b||_{H^{s-1}}).
\end{align*}
Integrating \eqref{eq1} over $[0,t]$, multiplying the inequality above by $2^{2j{s}}$ and summing over $j\geq -1$, we have
\bal\label{eq3}
||\delta u||^2_{H^{s}}+||\delta b||^2_{H^{s}}&\leq ||\delta u_0||^2_{H^{s}}+||\delta b_0||^2_{H^{s}}\nonumber\\& \quad +\int^t_0\mathrm{A}'(\tau)(||\delta u||^2_{H^{s}}+||\delta b||^2_{H^{s}})\dd \tau \nonumber
\\& \quad +C\int^t_0(||u^2||^2_{H^{s+1}}+||b^2||^2_{H^{s+1}})(||\delta u||^2_{H^{s-1}}+||\delta b||^2_{H^{s-1}}) \dd \tau.
\end{align}
This complete the proof of this lemma.
\end{proof}

\begin{lemm}\label{le-2}
Let $d\geq 2$ and $s>1+\frac d2$. Suppose that $(u,b)\in \mathcal{C}([0,T];H^{s+2})$ is the solution of \eqref{MHD1} with initial data
$(u_0,b_0)$ and $(v,c)\in \mathcal{C}([0,T];H^{s+2})$ is the solution of \eqref{ideal-MHD} with initial data $(v_0,c_0)$. Denote $\omega=u-v$ and $a=b-c$. Then, we have
\bbal
||\omega(t)||^2_{H^{s-1}}+||a(t)||^2_{H^{s-1}}&\leq C\Big(||\omega_0||^2_{H^{s-1}}+||a_0||^2_{H^{s-1}}\\& \quad +\mu^2 \int^t_0||u||^2_{H^{s+1}}\dd \tau+\nu^2 \int^t_0||b||^2_{H^{s+1}}\dd \tau\Big)e^{\mathrm{B}(t)},
\end{align*}
and
\bbal
||\omega(t)||^2_{H^{s}}+||a(t)||^2_{H^{s}}&\leq C\Big(||\omega_0||^2_{H^{s}}+||a_0||^2_{H^{s}}+\mu^2 \int^t_0||u||^2_{H^{s+2}}\dd \tau+\nu^2 \int^t_0||b||^2_{H^{s+2}}\dd \tau\\& \quad +\int^t_0(||v||^2_{H^{s+1}}+||c||^2_{H^{s+1}})(||w||^2_{H^{s-1}}+||a||^2_{H^{s-1}}) \dd \tau\Big)e^{\mathrm{B}(t)},
\end{align*}
with
\bbal
\mathrm{B}(t)=C\int^t_0(1+||u||_{H^s}+||v||_{H^s}+||b||_{H^s}+||c||_{H^s})\dd \tau.
\end{align*}
\end{lemm}
\begin{proof}
It is easy to show that
\beq\begin{cases}
\partial_t\omega+u\cdot\nabla \omega+\omega\cdot\na v+\nabla P=b\cdot \nabla a+a\cd \na c+\mu\De u, \\
\partial_ta+u\cdot \nabla a+\omega \cd \na c=b\cdot\nabla \omega+a\cd \na v+\nu\De b,\\
\mathrm{div} \omega=\mathrm{div} a=0,\quad (\omega,a)|_{t=0}=(\omega_0,a_0).
\end{cases}\eeq
The similar argument as in \eqref{eq2}, we have
\bbal
||\omega||^2_{H^{s-1}}+||a||^2_{H^{s-1}}&\leq ||\omega_0||^2_{H^{s-1}}+||a_0||^2_{H^{s-1}}+C\mu^2\int^t_0||u||^2_{H^{s+1}}\dd \tau+C\nu^2\int^t_0||b||^2_{H^{s+1}}\dd \tau\\& \quad +\int^t_0\mathrm{B}'(\tau)(||\omega||^2_{H^{s-1}}+||a||^2_{H^{s-1}})\dd \tau.
\end{align*}
The similar argument as in \eqref{eq3}, we have
\bbal
||\omega||^2_{H^{s}}+||a||^2_{H^{s}}&\leq ||\omega_0||^2_{H^{s}}+||a_0||^2_{H^{s}}+C\mu^2\int^t_0||u||^2_{H^{s+2}}\dd \tau+C\nu^2\int^t_0||b||^2_{H^{s+2}}\dd \tau\\& \quad +\int^t_0\mathrm{B}'(\tau)(||\omega||^2_{H^{s}}+||a||^2_{H^{s}})\dd \tau
\\& \quad +C\int^t_0(||v||^2_{H^{s+1}}+||c||^2_{H^{s+1}})(||\omega||^2_{H^{s-1}}+||a||^2_{H^{s-1}}) \dd \tau.
\end{align*}
This along with Gronwall's inequality completes the proof of the lemma.
\end{proof}

\begin{lemm}\label{le-3}
Let $d\geq 2$ and $(s,p,r)$ satisfies \eqref{eq:spr}. Suppose that $(u^1,b^1)\in \mathcal{C}([0,T];B^{s}_{p,r})$ and $(u^2,b^2)\in \mathcal{C}([0,T];B^{s+1}_{p,r})$ are two solutions of \eqref{MHD2} with initial data $(u^1_0,b^1_0)$ and $(u^2_0,b^2_0)$ respectively. Denote $\delta u=u^1-u^2$ and $\delta b=b^1-b^2$. Then, we have
\bbal
||\delta u(t)||_{B^{s-1}_{p,r}}+||\delta b(t)||_{B^{s-1}_{p,r}}\leq C(||\delta u_0||_{B^{s-1}_{p,r}}+||\delta b_0||_{B^{s-1}_{p,r}})e^{\bar{\mathrm{A}}(t)},
\end{align*}
and
\bbal
||\delta u(t)||_{B^{s}_{p,r}}+||\delta b(t)||_{B^{s}_{p,r}}&\leq C \Big(||\delta u_0||_{B^{s}_{p,r}}+||\delta b_0||_{B^{s}_{p,r}}\\& \quad +\int^t_0(||u^2||_{B^{s+1}_{p,r}}+||b^2||_{B^{s+1}_{p,r}})(||\delta u||_{B^{s-1}_{p,r}}+||\delta b||_{B^{s-1}_{p,r}}) \dd \tau\Big)e^{\bar{\mathrm{A}}(t)},
\end{align*}
with
\bbal
\bar{\mathrm{A}}(t)=C\int^t_0(||u^1||_{B^s_{p,r}}+||u^2||_{B^s_{p,r}}+||b^1||_{B^s_{p,r}}+||b^2||_{B^s_{p,r}})\dd \tau.
\end{align*}
\end{lemm}
\begin{proof}
Denote $\bar{u}^i=u^i+b^i$ and $\bar{b}^i=u^i-b^i$ for $i=1,2$. Then, setting $\delta \bar{u}=\bar{u}^1-\bar{u}^2$ and $\delta \bar{b}=\bar{b}^1-\bar{b}^2$, we get
\beq\bca
\pa_t\delta \bar{u}+\bar{b}^1\cd \na \delta \bar{u}+\delta \bar{b}\cd \na \bar{u}^2-\mu\De\delta\bar{u}+\na P_1=0,\\
\pa_t\delta \bar{b}+\bar{u}^1\cd \na \delta \bar{b}+\delta \bar{u}\cd \na \bar{b}^2-\mu\De\delta\bar{b}+\na P_2=0,\\
\D \delta \bar{u}=\D \delta \bar{u}=0, \quad (\delta \bar{u},\delta \bar{b})|_{t=0}=(\delta \bar{u}_0,\delta \bar{b}_0).
\eca\eeq
It follows from Lemma \ref{lem:P} that
\bbal
||\na P_1||_{B^{s-1}_{p,r}}+||\na P_2||_{B^{s-1}_{p,r}}\leq \bar{\mathrm{A}}'(||\delta \bar{u}||_{B^{s-1}_{p,r}}+||\delta \bar{b}||_{B^{s-1}_{p,r}}),
\end{align*}
and
\bbal
||\na P_1||_{B^{s}_{p,r}}+||\na P_2||_{B^{s}_{p,r}}\leq \bar{\mathrm{A}}'(||\delta \bar{u}||_{B^{s}_{p,r}}+||\delta \bar{b}||_{B^{s}_{p,r}}).
\end{align*}
According to Proposition \ref{pr:TDe}, we have
\bbal
||\delta \bar{u}||_{B^{s-1}_{p,r}}+||\delta \bar{b}||_{B^{s-1}_{p,r}}\leq ||\delta \bar{u}_0||_{B^{s-1}_{p,r}}+||\delta \bar{b}_0||_{B^{s-1}_{p,r}}+\int^t_0\bar{\mathrm{A}}'(||\delta \bar{u}||_{B^{s-1}_{p,r}}+||\delta \bar{b}||_{B^{s-1}_{p,r}})\dd \tau,
\end{align*}
which along with Gronwall's inequality yield
\bbal
||\delta \bar{u}||_{B^{s-1}_{p,r}}+||\delta \bar{b}||_{B^{s-1}_{p,r}}\leq (||\delta \bar{u}_0||_{B^{s-1}_{p,r}}+||\delta \bar{b}_0||_{B^{s-1}_{p,r}})e^{\bar{\mathrm{A}}(t)}.
\end{align*}
Similarly, we also have
\bbal
||\delta \bar{u}||_{B^{s}_{p,r}}+||\delta \bar{b}||_{B^{s}_{p,r}}\leq & ||\delta \bar{u}_0||_{B^{s}_{p,r}}+||\delta \bar{b}_0||_{B^{s}_{p,r}}+\int^t_0\bar{\mathrm{A}}'(||\delta \bar{u}||_{B^{s}_{p,r}}+||\delta \bar{b}||_{B^{s}_{p,r}})\dd \tau
\\&+C\int^t_0(||u^2||_{B^{s+1}_{p,r}}+||b^2||_{B^{s+1}_{p,r}})(||\delta \bar{u}||_{B^{s-1}_{p,r}}+||\delta \bar{b}||_{B^{s-1}_{p,r}}) \dd \tau,
\end{align*}
which along with Gronwall's inequality yields
\bbal
||\delta \bar{u}||_{B^{s}_{p,r}}+||\delta \bar{b}||_{B^{s}_{p,r}}&\leq \Big(||\delta \bar{u}_0||_{B^{s}_{p,r}}+||\delta \bar{b}_0||_{B^{s}_{p,r}}\\& \quad +C\int^t_0(||u^2||_{B^{s+1}_{p,r}}+||b^2||_{B^{s+1}_{p,r}})(||\delta \bar{u}||_{B^{s-1}_{p,r}}+||\delta \bar{b}||_{B^{s-1}_{p,r}}) \dd \tau\Big)e^{\bar{\mathrm{A}}(t)}.
\end{align*}
It is easy to see for all $\sigma \in \R$,
\bal\label{li}
||\delta \bar{u}||_{B^{\sigma}_{p,r}}+||\delta \bar{b}||_{B^{\sigma}_{p,r}}\simeq ||\delta u||_{B^{\sigma}_{p,r}}+||\delta b||_{B^{\sigma}_{p,r}}.
\end{align}
Therefore, combining the above inequalities, we obtain the corresponding results.
\end{proof}

\begin{lemm}\label{le-4}
Let $d\geq 2$ and $(s,p,r)$ satisfies \eqref{eq:spr}. Suppose that $(u,b)\in \mathcal{C}([0,T];B^{s+2}_{p,r})$ is the solution of \eqref{MHD2} with initial data $(u_0,b_0)$ and $(v,c)\in \mathcal{C}([0,T];B^{s+2}_{p,r})$ is the solution of \eqref{ideal-MHD} with initial data $(v_0,c_0)$. Denote $\omega=u-v$ and $a=b-c$. Then, we have
\bbal
||\omega(t)||_{B^{s-1}_{p,r}}+||a(t)||_{B^{s-1}_{p,r}}&\leq C \Big(||\omega_0||_{B^{s-1}_{p,r}}+||a_0||_{B^{s-1}_{p,r}}\\& \quad +\mu \int^t_0||u||_{B^{s+1}_{p,r}}\dd \tau+\mu \int^t_0||b||_{B^{s+1}_{p,r}}\dd \tau\Big)e^{\bar{\mathrm{B}}(t)},
\end{align*}
and
\bbal
||\omega(t)||_{B^{s}_{p,r}}+||a(t)||_{B^{s}_{p,r}}&\leq C \Big(||\omega_0||_{B^{s}_{p,r}}+||a_0||_{B^{s}_{p,r}}+\mu \int^t_0||u||_{B^{s+2}_{p,r}}\dd \tau+\mu \int^t_0||b||_{B^{s+2}_{p,r}}\dd \tau\\& \quad +\int^t_0(||v||_{B^{s+1}_{p,r}}+||c||_{B^{s+1}_{p,r}})(||w||_{B^{s-1}_{p,r}}+||a||_{B^{s-1}_{p,r}}) \dd \tau\Big)e^{\bar{\mathrm{B}}(t)},
\end{align*}
with
\bbal
\bar{\mathrm{B}}(t)=C\int^t_0(||u||_{B^s_{p,r}}+||v||_{B^s_{p,r}}+||b||_{B^s_{p,r}}+||c||_{B^s_{p,r}})\dd \tau.
\end{align*}
\end{lemm}
\begin{proof}
Denote $\bar{u}=u+b$, $\bar{b}=u-b$, $\bar{v}=v+c$ and $\bar{c}=v-c$. Then, setting $\bar{\omega}=\bar{u}-\bar{v}$ and $\bar{a}=\bar{b}-\bar{c}$, we have
\beq\bca
\pa_t\bar{\omega}+\bar{b}\cd \na \bar{\omega}+\bar{a}\cd \na \bar{v}+\na \bar{P}_1=\mu\De\bar{u},\\
\pa_t \bar{a}+\bar{u}\cd \na \bar{a}+\bar{\omega}\cd \na \bar{c}+\na \bar{P}_2=\mu \De \bar{b},\\
\D \bar{\omega}=\D \bar{a}=0, \quad (\bar{\omega},\bar{a})|_{t=0}=(\bar{\omega}_0,\bar{a}_0).
\eca\eeq
The similar arguments as in Lemma \ref{le-3}, we have
\bbal
||\bar{\omega}||_{B^{s-1}_{p,r}}+||\bar{a}||_{B^{s-1}_{p,r}}&\leq ||\bar{\omega}_0||_{B^{s-1}_{p,r}}+||\bar{a}_0||_{B^{s-1}_{p,r}} +C\mu \int^t_0||u||_{B^{s+1}_{p,r}}\dd \tau+C\mu \int^t_0||b||_{B^{s+1}_{p,r}}\dd \tau \\& \quad +\int^t_0\bar{\mathrm{B}}'(\tau)(||\bar{\omega}||_{B^{s-1}_{p,r}}+||\bar{a}||_{B^{s-1}_{p,r}})\dd \tau,
\end{align*}
and
\bbal
||\bar{\omega}||_{B^{s}_{p,r}}+||\bar{a}||_{B^{s}_{p,r}}&\leq ||\bar{\omega}_0||_{B^{s}_{p,r}}+||\bar{a}_0||^2_{B^{s}_{p,r}}+C\mu \int^t_0||u||_{B^{s+2}_{p,r}}+C\mu \int^t_0||b||_{B^{s+2}_{p,r}}\dd \tau\\& \quad +C\int^t_0(||v||_{B^{s+1}_{p,r}}+||c||_{B^{s+1}_{p,r}})(||\bar{\omega}||_{B^{s-1}_{p,r}}+||\bar{a}||_{B^{s-1}_{p,r}}) \dd \tau\\& \quad +\int^t_0\bar{\mathrm{B}}'(\tau)(||\bar{\omega}||_{B^{s}_{p,r}}+||\bar{a}||_{B^{s}_{p,r}})\dd \tau.
\end{align*}
This along with Gronwall's inequality and \eqref{li} completes this proof of this lemma.
\end{proof}

\section{Proof of Theorem 1.1 and Theorem 1.2}
In order to simplify the notation, we denote $(u,b)$ by $(u^0,b^0)$ satisfies $\eqref{revist1-MHD}$ or \eqref{revist2-MHD} with $\mu_0=\nu_0=0$.  First, we state the proof of Theorem 1.1.

\noindent\textbf{Proof of Theorem 1.1.} First, according to classical results, there exist a positive $T_n>0$ such that \eqref{revist1-MHD} have a solution $(u^n,b^n)\in \mathcal{C}([0,T_n);H^s)$. Indeed, by Lemma \ref{le-1} and \eqref{Li-yin}, we have
\bbal
||u^n||^2_{H^s}+||b^n||^2_{H^s}\leq (||u^n_0||^2_{H^s}+||b^n_0||^2_{H^s})
e^{C\int^t_0(||u^n||_{C^{0,1}}+||b^n||_{C^{0,1}})\dd \tau}.
\end{align*}
Denote $R=\sup\limits_{n\geq 0}(||u^n_0||_{H^s}+||b^n_0||_{H^s})$. Therefore, by continuity arguments, there exists a positive $T=T(s,d,R)$ such that
\bbal
||u^n||^2_{L^\infty_T(H^s)}+||b^n||^2_{L^\infty_T(H^s)}\leq C(||u^n_0||^2_{H^s}+||b^n_0||^2_{H^s})\leq C.
\end{align*}
Moreover, for all $\gamma>s$, we have
\bbal
||u^n||^2_{H^\gamma}+||b^n||^2_{H^\gamma}&\leq (||u^n_0||^2_{H^\gamma}+||b^n_0||^2_{H^\gamma})
e^{C\int^t_0(||u^n||_{C^{0,1}}+||b^n||_{C^{0,1}})\dd \tau}\\&\leq C(||u^n_0||^2_{H^\gamma}+||b^n_0||^2_{H^\gamma}).
\end{align*}
Let $(u^n_j,b^n_j)\in \mathcal{C}([0,T];H^s)$ be the solution of
\beq\begin{cases}
\partial_tu^n_j+u^n_j\cdot\nabla u^n_j+\mu_n\De u^n_j+\nabla P_{1,j}=b^n_j\cdot \nabla b^n_j, \\
\partial_tb^n_j+u^n_j\cdot \nabla b^n_j+\nu_n\De b^n_j=b^n_j\cdot\nabla u^n_j,\\
\mathrm{div} u^n_j=\mathrm{div} b^n_j=0,\quad (u^n_j,b^n_j)|_{t=0}=S_j(u^n_0,b^n_0).
\end{cases}\eeq
Then, according to Lemma \ref{le-1}, we have
\bbal
||u^n_j-u^n||^2_{H^{s-1}}+||b^n_j-b^n||^2_{H^{s-1}}\leq  C(||(\mathrm{Id}-S_j)u^n_0||^2_{H^{s-1}}+||(\mathrm{Id}-S_j)b^n_0||^2_{H^{s-1}}),
\end{align*}
which along with the fact that $||u^n_j||_{H^{s+1}}+||b^n_j||_{H^{s+1}}\leq C2^j$ leads to
\bbal
& \quad \ ||u^n_j-u^n||^2_{H^{s}}+||b^n_j-b^n||^2_{H^{s}}\\&\leq C(||(\mathrm{Id}-S_j)u^n_0||^2_{H^{s}}+||(\mathrm{Id}-S_j)b^n_0||^2_{H^{s}}\\& \quad \ +\int^t_0(||u^n_j||^2_{H^{s+1}}+||b^n_j||^2_{H^{s+1}})(||u^n_j-u^n||^2_{H^{s-1}}+||b^n_j-b^n||^2_{H^{s-1}}) \dd \tau)
\\&\leq C(||(\mathrm{Id}-S_j)u^n_0||^2_{H^{s}}+||(\mathrm{Id}-S_j)b^n_0||^2_{H^{s}}
\\& \quad \ +2^{2j}||(\mathrm{Id}-S_j)u^n_0||^2_{H^{s-1}}+2^{2j}||(\mathrm{Id}-S_j)b^n_0||^2_{H^{s-1}})
\\& \leq C(||(\mathrm{Id}-S_j)u^n_0||^2_{H^{s}}+||(\mathrm{Id}-S_j)b^n_0||^2_{H^{s}}).
\end{align*}
Using the similar argument and Lemma \ref{le-2}, we can show that
and
\bbal
||u^n_j-u_j||^2_{H^{s}}+||b^n_j-b_j||^2_{H^{s}}\leq C2^{4j}(||u^n_0-u_0||^2_{H^s}+||b^n_0-b_0||^2_{H^s}+\mu^2_n+\nu^2_n).
\end{align*}
Therefore, combing the above inequalities, we obtain
\bbal
& \quad \ ||u^n-u||^2_{L^\infty_T(H^s)}+||b^n-b||^2_{L^\infty_T(H^s)}\\&\leq ||u^n_j-u_j||^2_{L^\infty_T(H^s)}+||b^n_j-b_j||^2_{L^\infty_T(H^s)}
\\&\quad  +||u^n_j-u^n||^2_{L^\infty_T(H^s)}+||b^n_j-b^n||^2_{L^\infty_T(H^s)}
\\&\quad +||u_j-u||^2_{L^\infty_T(H^s)}+||b_j-b||^2_{L^\infty_T(H^s)}
\\& \leq  C(||(\mathrm{Id}-S_j)u^n_0||^2_{H^{s}}+||(\mathrm{Id}-S_j)b^n_0||^2_{H^{s}}+||(\mathrm{Id}-S_j)u_0||^2_{H^{s}}\\& \quad \ +||(\mathrm{Id}-S_j)b_0||^2_{H^{s}}+
2^{4j}||u^n_0-u_0||^2_{H^s}+2^{4j}||b^n_0-b_0||^2_{H^s}+2^{4j}\mu^2_n+2^{4j}\nu^2_n)
\\& \leq  C(||(\mathrm{Id}-S_j)u_0||^2_{H^{s}}+||(\mathrm{Id}-S_j)b_0||^2_{H^{s}}\\& \quad \ +
2^{4j}||u^n_0-u_0||^2_{H^s}+2^{4j}||b^n_0-b_0||^2_{H^s}+2^{4j}\mu^2_n+2^{4j}\nu^2_n).
\end{align*}
This completes the proof of Theorem 1.1. \\

Now, we will prove Theorem 1.2.

\noindent\textbf{Proof of Theorem 1.2.} First, according to classical results, there exist a positive $T_n>0$ such that \eqref{revist2-MHD} have a solution $(u^n,b^n)\in \mathcal{C}([0,T_n);B^s_{p,r})$. Indeed, by Lemma \ref{le-3} and \eqref{Li-yin}, we have
\bbal
||u^n||_{B^s_{p,r}}+||b^n||_{B^s_{p,r}}\leq (||u^n_0||_{B^s_{p,r}}+||b^n_0||_{B^s_{p,r}})
e^{C\int^t_0(||u^n||_{C^{0,1}}+||b^n||_{C^{0,1}})\dd \tau}.
\end{align*}
Denote $\bar{R}=\sup\limits_{n\geq 0}(||u^n_0||_{B^s_{p,r}}+||b^n_0||_{B^s_{p,r}})$. Therefore, by continuity arguments, there exists a positive $T=T(s,p,r,d,\bar{R})$ such that
\bbal
||u^n||_{L^\infty_T(B^s_{p,r})}+||b^n||_{L^\infty_T(B^s_{p,r})}\leq C(||u^n_0||_{B^s_{p,r}}+||b^n_0||_{B^s_{p,r}}).
\end{align*}
Moreover, for all $\gamma>s$, we have
\bbal
||u^n||_{B^\gamma_{p,r}}+||b^n||_{B^\gamma_{p,r}}&\leq (||u^n_0||_{B^\gamma_{p,r}}+||b^n_0||_{B^\gamma_{p,r}})
e^{C\int^t_0(||u^n||_{C^{0,1}}+||b^n||_{C^{0,1}})\dd \tau}\\&\leq C(||u^n_0||_{B^\gamma_{p,r}}+||b^n_0||_{B^\gamma_{p,r}}).
\end{align*}
Now, we set $(u^n_j,b^n_j)$ satisfies the following system:
\beq\begin{cases}
\partial_tu^n_j+u^n_j\cdot\nabla u^n_j+\mu_n\De u^n_j+\nabla P_j=b^n_j\cdot \nabla b^n_j, \\
\partial_tb^n_j+u^n_j\cdot \nabla b^n_j+\mu_n\De b^n_j=b^n_j\cdot\nabla u^n_j,\\
\mathrm{div} u^n_j=\mathrm{div} b^n_j=0,\quad (u^n_j,b^n_j)|_{t=0}=(S_ju^n_0,S_jb^n_0).
\end{cases}\eeq
Similar proof as in Theorem 1.1, by Lemmas \ref{le-3}-\ref{le-4}, we have
\bbal
& \quad \ ||u^n_j-u^n||_{L^\infty_T(B^s_{p,r})}+||b^n_j-b^n||_{L^\infty_T(B^s_{p,r})}\\&\leq C(||(\mathrm{Id}-S_j)u^n_0||_{B^s_{p,r}}+||(\mathrm{Id}-S_j)b^n_0||_{B^s_{p,r}}),
\end{align*}
and
\bbal
& \quad \ ||u^n_j-u_j||_{L^\infty_T(B^s_{p,r})}+||b^n_j-b_j||_{L^\infty_T(B^s_{p,r})}\\&\leq C2^{2j}(||u^n_0-u_0||_{B^s_{p,r}}+||b^n_0-b_0||_{B^s_{p,r}}+\mu_n).
\end{align*}
Then, we can show that
\bbal
&\quad \ ||u^n-u||_{L^\infty_T(B^s_{p,r})}+||b^n-b||_{L^\infty_T(B^s_{p,r})}\\& \leq
C(||(\mathrm{Id}-S_j)u_0||_{B^s_{p,r}}+||(\mathrm{Id}-S_j)b_0||_{B^s_{p,r}}\\& \quad \ +2^{2j}||u^n_0-u_0||_{B^s_{p,r}}+2^{2j}||b^n_0-b_0||_{B^s_{p,r}}+2^{2j}\mu_n).
\end{align*}
This completes the proof of Theorem 1.2.

\vspace*{1em}
\noindent\textbf{Acknowledgements.} This work was
partially supported by NNSFC (No. 11271382), RFDP (No. 20120171110014), MSTDF (No. 098/2013/A3), Guangdong Special Support Program (No. 8-2015) and the key project of NSF of Guangdong Province (No. 1614050000014).

\end{document}